\documentclass[12pt,reqno]{amsart}
\usepackage{amssymb,amsmath}
\def\({\left(}
\def\){\right)}
\def\Nx{\nabla}
\def\divv{\operatorname{div}}
\def\eb{\varepsilon}

\def\eb{\varepsilon}

\def\dist{{\rm dist}}

\def\R {\mathbb{R}}

\def \and{\qquad\text{and}\qquad}

\def\Dt{\partial_t}
\def\Dx{\Delta}
\newtheorem{proposition}{Proposition}[section]
\newtheorem{theorem}[proposition]{Theorem}
\newtheorem{corollary}[proposition]{Corollary}
\newtheorem{lemma}[proposition]{Lemma}
\theoremstyle{definition}
\newtheorem{definition}[proposition]{Definition}
\newtheorem{remark}[proposition]{Remark}

\numberwithin{equation}{section}
\def \no#1#2#3 {{\bf #1} (#3), #2.}
\def \eds#1#2#3 {#1, #2, #3.}
\title[Attractor for the Brinkman-Forchheimer Equations]
{Smooth Attractors for the Brinkman-Forchheimer equations with  fast growing nonlinearities}
\author[]
{Varga K. Kalantarov and Sergey Zelik}

\address{(V.K.Kalantarov) Department of mathematics,
\newline\indent Ko{\c c} University, Rumelifeneri Yolu, Sariyer, Istanbul, Turkey
}
\address{(S.K.Zelik) Department of mathematics, \newline \indent University of Surrey
Guildford, GU2 7XH, UK}
\keywords{Brinkmann-Forchheimer equations, attractors, maximal regularity, nonlinear localization}
 
\begin{document}

\begin{abstract}{We prove the existence of regular dissipative solutions and global attractors for the 3D Brinkmann-Forchheimer equations
 with the nonlinearity of an arbitrary polynomial growth rate. In order to obtain this result, we prove the maximal
 regularity estimate for the corresponding semi-linear stationary Stokes problem using some modification of the
  nonlinear localization technique. The applications of our results to the Brinkmann-Forchheimer equation with
  the Navier-Stokes inertial term are also considered.}
\end{abstract}
\subjclass{35B40, 35B41, 35Q35}
\maketitle

\section{Introduction}\label{s0}

We study  the
Brinkman-Forchheimer (BF) equations in the following form:
%$$
\begin{equation}\label{bf1}
\begin{cases}
\Dt u-\Delta u +f(u)+\nabla p =g, \ \ \divv u=0, \\
u\big|_{\partial\Omega}=0,\ \ u\big|_{t=0}=u_0.
\end{cases}
\end{equation}
%$$
Here $\Omega \subset \R^3$ is an open, bounded domain  with $C^2$
boundary $\partial \Omega$, $ g=g(x)=(g_1,g_2,g_3)$  is a given
function, $u=(u_1,u_2,u_3)$ is the fluid velocity vector, $p$ is the
pressure and $f$ is a given nonlinearity.
\par
The BF equations are used to describe the fluid flow in a saturated porous media, see \cite{D,str} and references therein.
The typical example for
$f$ is the following one:
%$$
\begin{equation}\label{nf}
 f(u)=au+b|u|^{r-1}u, \ r \in [1,\infty),
\end{equation}
%$$
where $a\in R$ and $b>0$ are the Darcy and Forchheimer coefficients respectively (the original Brinkman-Forchheimer model corresponds to the choice $r=2$, more complicated nonlinear terms ($r\ne2$) appear, e.g.,
in the theory of non-Newtonian fluids, see \cite{Sh}). Note also that the analogous equations are
used in the study of  tidal dynamics (see \cite{Go},\cite{Lik}).
\par
 Number of papers is devoted to the mathematical study of
of the BF equations, for instance, continuous dependence on
changes in Brinkman and Forchheimer coefficients and convergence of
solutions of BF equations to the solution of the Forchheimer
equation
$$
\Dt u+f(u)+\Nx p=g,  \\ \divv u=0,
$$
as the viscosity tends to zero have been established in
\cite{CKU,CKU1,LL,PS,str} (see also  references therein), and
the long-time behavior of solutions for \eqref{bf1} has been studied in terms of global attractors in   \cite{QY}
\cite{Ug} and \cite{WS}. However, to the best of our knowledge, only the case of the so-called subcritical growth rate of the nonlinearity $f$
 ($r\le3$ in \eqref{nf}) has been considered in the literature.
\par
The main aim of the present paper is to remove this growth restriction and  verify the global existence, uniqueness and dissipativity of smooth solutions of the BF equations for  the large class of nonlinearites $f$ of the arbitrary growth exponent $r\ge1$.
\par
Namely, we  assume that $f\in C^2(\R^3,\R^3)$
satisfies the following conditions:
%$$
\begin{equation}\label{0.fcond}
\begin{cases}
1)\ \ f'(u)v.v\geq (-K+\kappa|u|^{r-1})|v|^2, \ \forall u,v \in \R^3,\\
2)\ \   |f'(u)|\le C(1+|u|^{r-1}),\ \forall u \in \R^3,
\end{cases}
\end{equation}
%$$
where  $K,C,\kappa$ are some positive constants, $r\ge1$ and $u.v$ stands for the standard inner product in $\R^3$.
\par
Our key technical tool is the maximal regularity result for the
stationary problem
%$$
\begin{equation}
-\Dx w+f(w)+\Nx p=g,\ \ \divv w=0,\ u\big|_{\partial\Omega}=0
\end{equation}
%$$
which claims that the solution $w$ belongs to $H^2$ if $g\in L^2$. This result is straightforward
 for the case of periodic boundary conditions (it follows via the multiplication
 of the equation by $\Dx w$ and integrating by parts). However, for the case of Dirichlet
 boundary conditions it is far from being immediate since the additional uncontrollable boundary terms arise
  after the multiplication of the equation by $\Dx w$ and integrating by parts. Following the approach developed in \cite{KZ}, we
  overcome this problem using some kind of {\it nonlinear} localization technique, see Apendix below.
\par
In addition, we apply our maximal regularity result in order to establish the existence of smooth solutions for the so-called
convective BF equations:
%$$
\begin{equation}\label{0.ns}
\begin{cases}
\Dt u+(u,\Nx)u-\Delta u +f(u)+\nabla p =g, \ \ \divv u=0,\\
u\big|_{\partial\Omega}=0,\ \ u\big|_{t=0}=u_0
\end{cases}
\end{equation}
%$$
under the assumption \eqref{nf} with $r>3$. Note that the case $f=0$
corresponds to the classical Navier-Stokes problem where the
existence of smooth solutions is an open problem. However, as also
known (see \cite{RZh}) the sufficiently strong nonlinearity $f$
produces some kind of regularizing effect. Again, in contrast to the
previous works,  no upper bounds for the exponent $r$ are posed
here.
\par
The paper is organized as follows. A number of a priori estimates which are necessary to handle equation \eqref{bf1} is given
in Section \ref{s1}. Existence, uniqueness and regularity of solutions for the BF equation
 as well as the existence of the associated global attractor are established in Section \ref{s2}. These results are
  extended to the case of  convective BF equations \eqref{0.ns} in Section \ref{s3}.
   Finally, the crucial maximal regularity results for the stationary equations \eqref{bf1} and \eqref{0.ns}
    are obtained in Appendix.

%\noindent We denote by  $\lambda_1$ the first eigenvalue of the
%Stokes operator $A$.
% $ V:= (H_0^1(\Omega))^3\cap H$
% is the Hilbert space with
%the norm $\|u\|_V=\|\nabla u\|$. \\

%\noindent We will use also the following
%well-known inequalities:\\

%\noindent{\it Poincar\'{e} inequality.}
%\begin{equation}\label{Poi}
%\|u\|\leq \lambda_1^{-1/2}\|u\|_1, \ \ \forall u \in V,
%\end{equation}
%where $\lambda_1$ is the first eigenvalue of the Stokes operator
%under the homogeneous Dirichlet boundary condition.\\

% \noindent {\it Interpolation inequality.}  If $u\in L_p(\Om)\cap
% \L^q(\Om), 1\leq p \leq q\leq \infty,$ then $u\in L^r(\Om)$ for all
% $r\in[p,q]$, and the following inequlity holds
%\begin{equation}\label{int}
%\|u\|_{L^r}\leq \|u\|_{L^p}^{\alpha}\|u\|_{L^q}^{1-\alpha}, \ \
%\alpha \in[0,1], \ \frac 1r=\frac{\alpha}p+\frac{1-\alpha}q.
%\end{equation}

%\noindent{\it Monotonicity inequalities.} (see, e.g.,
%\cite{Li},\cite{GGZ})
%\begin{equation}\label{mn1}
%|\chi_2|^{r+1} - |\chi_1|^{r+1}\geq
%(r+1)|\chi_1|^{r-1}\chi_1(\chi_2-\chi_1)+
%\frac{|\chi_2-\chi_1|^r}{2^{r}-1},  \  \ \forall \chi_1,\chi_2\in
%\R^n, \ r\geq 1.
%\end{equation}
%If $f(\cdot): \R^n\rightarrow \R^n$ is monotone operator with
%potential function $F(\cdot):\R^n \rightarrow \R^1$ then

%\begin{equation}\label{mn2}
%F(\chi_2) - F(\chi_1) \geq f(\chi_1)\cdot(\chi_2-\chi_1), \ \
%\forall \chi_1,\chi_2 \in \R^n.
%\end{equation}

\section{A priori estimates.}\label{s1}
In this section, we obtain a number of a priori estimates for the
solutions of the problem \eqref{bf1} assuming that the sufficiently
regular solution $(u,p)$ of this equation is given. These estimates
will be used in the next sections in order to establish the
existence and uniqueness of solution, their regularity, etc.
\par
We start with introducing the standard notations. As usual, we denote by $W^{l,p}(\Omega)$ the Sobolev space of all functions
whose distributional derivatives up to order $l$ belong to $L^p(\Omega)$. The Hilbert spaces $W^{l,2}(\Omega)$ will
 be also denoted by $H^l(\Omega)$.
\par
For the vector valued functions $v =(v_1,v_2,v_3),$ and $ u=(u_1,u_2,u_3)$ we denote
by $(u,v)$ the standard inner product in $[L^2(\Omega)]^3$:
$$
(u,v):=\sum\limits_{j=1}^3(v_j,u_j)_{L^2(\Omega)}, \
$$
and write $\|\Nx u\|^2_{L^2}$ instead of $\sum\limits_{i=1}^3\|\Nx
u_i\|^2_{L^2}$. In the sequel, where it does
 not lead to misunderstandings, we will also use the notation $H^l(\Omega)$ and $W^{l,p}(\Omega)$ for
  the spaces of vector valued functions $[H^l(\Omega)]^3$ and $[W^{l,p}(\Omega)]^3$ respectively.
\par
 As usual, we set
 $$
  \mathcal{V}:=\left\{v \in
(C_0^{\infty}(\Omega))^3: \ \divv v=0\right\},
 $$
 and denote by $H$ and $H_1=V$ the closure
of $\mathcal{V}$ in $L_2(\Omega)$ and $H^1(\Omega)$ topology
respectively. And, more generally, $H^s:=D(A^{s/2})$, where
$A:=\Pi\Dx$ and $\Pi$ is the classical Helmholz-Leray orthogonal
projection in $L^2(\Omega)$ onto the space $H$. In particular, since
$\Omega$ is smooth and bounded, we have
$$
H=\{u\in L^2(\Omega),\ \divv u=0, (u,n)\big|_{\partial\Omega}=0\},\ H^1:=H\cap H^1_0(\Omega),\ \ H^2=H^1\cap H^2(\Omega),
$$
see e.g. \cite{La}.
\par
The next lemma gives  the usual energy estimate for the BF equation.

\begin{lemma} \label{Lem1.en} Let $(u,p)$ be a sufficiently smooth solution of problem \eqref{bf1}. Then the following
estimate holds:
%$$
\begin{equation}\label{1.en}
\|u(t)\|_{L^2}^2+\int_t^{t+1}\left[\|\Nx
u(s)\|^2_{L^2}+\|u(s)\|^{r+1}_{L^{r+1}}\right] \ ds\le
C\|u(0)\|_{L^2}^2e^{-\alpha t}+C(1+\|g\|^2_{L^2}),
\end{equation}
%$$
where the positive constants $C$ and $\alpha$ are independent of $t$ and the concrete choice of the solution $(u,p)$.
\end{lemma}
\begin{proof} Indeed, multiplying equation \eqref{bf1} by $u$, integrating over $x\in\Omega$, using that $f(u).u\ge -C+\kappa|u|^{r+1}$ and $(\Nx p,u)=(p,\divv u)=0$ and arguing in a standard way, we have
%$$
\begin{equation}
\frac12\Dt \|u(t)\|^2_{L^2}+\alpha\|u(t)\|_{H^1}^2+\alpha\|u(t)\|^{r+1}_{L^{r+1}}\le C(1+\|g\|^2_{L^2})
\end{equation}
%$$
for some positive $\alpha$ and $C$ which are independent of $u$ and
$t$. Applying the Gronwall inequality to the last estimate, we
derive \eqref{1.en} and finish the proof of the lemma.
\end{proof}

\begin{remark} \label{Rem1.bad} The standard (for the reaction-diffusion equations) next step in a priori estimates would
be the multiplication of equation \eqref{bf1} by $\Dx u$ (or $t\Dx
u$) and obtaining the dissipative estimate in $H^1$ together with
the $L^2\to H^1$ parabolic smoothing property. However, in our case,
this scheme looks not applicable since $\Dx
u\big|_{\partial\Omega}\ne0$ in general and the term with pressure
will not disappear. Multiplication by $\Pi\Dx u$ (where $\Pi$ is the
Helmholz-Leray projector to the divergent free vector fields) also
does not work due to the presence  of the non-linearity $f$ with
arbitrary growth rate. So, we have to skip this step and estimate
the $L^2$-norm of $\Dt u$ instead differentiating equation by $t$
and using the quasi-monotonicity of $f$. The $H^1$ (and $H^2$)
estimate will be obtained after that using the maximal regularity
theorem for the elliptic problem \eqref{st1}, see Appendix).
\end{remark}
The next simple corollary is, however, crucial for our method of proving the existence and dissipativity of the $H^2$-solutions.

\begin{corollary}\label{Cor1.dt} Let $(u,p)$ be a sufficiently regular solution of problem \eqref{bf1}.
   Then, the following estimate
holds:
%$$
\begin{equation}\label{1.dt}
\|\Dt u\|_{L^1([t,t+1],H^{-2})}\le Q(\|u(0)\|_{L^2})e^{-\alpha
t}+Q(\|g\|_{L^2}),
\end{equation}
%$$
where the monotone function $Q$ and the constant $C$ are independent of $t$ and $u$.
\end{corollary}
\begin{proof} Indeed, applying the Helmholz-Leray projector $\Pi$ to both sides of equation \eqref{bf1} and using that $\divv\Dt u=0$,
we arrive at
%$$
\begin{equation}\label{1.nopres}
\Dt u=Au-\Pi f(u)+\Pi g.
\end{equation}
%$$
Thanks to the growth restriction on $f$ and the control \eqref{1.en}, we have
$$
\|f(u)\|_{L^{r^*}([t,t+1],L^{r^*})}^r\le C\|u(0)\|^2_{L^2}e^{-\alpha t}+C(1+\|g\|^2_{L^2})
$$
with $r^*:=\frac{r+1}r$. Using now that the Helmholz-Leray projector $\Pi: L^{r^*}\to L^{r^*}$ together with the embedding
$L^{r^*}\subset H^{-2}$ (recall that $n=3$), we arrive at
$$
\|\Pi f(u)\|_{L^1([t,t+1],H^{-2})}\le Q(\|u(0)\|_{L^2})e^{-\alpha t}+Q(\|g\|_{L^2})
$$
for some monotone increasing function $Q$. This estimate, together with \eqref{1.nopres} and the control of $u$ given by the
energy estimate \eqref{1.en} give the desired estimate \eqref{1.dt} and finish the proof of the corollary.
\end{proof}

Let us now differentiate \eqref{bf1} with respect to time and denote
$v=\Dt u$. Then, this function solves
%$$
\begin{equation}\label{1.dteq}
\Dt v=\Dx v-f'(u)v+\Nx q,\ \ \divv v=0,\ \ v(0)=Au(0)-\Pi f(u(0))+\Pi g.
\end{equation}
%$$
Moreover, using the embedding $H^2\subset C$, we see that
%$$
\begin{equation}\label{1.dtini}
\|v(0)\|_{L^2}\le Q(\|u(0)\|_{H^2})+\|g\|_{L^2}
\end{equation}
%$$
and, therefore, the $L^2$-norm of the initial data for $v$ is under the control if $u(0)\in H^2$.
\par
The next Lemma gives the control of $v(t)$ for all $t\ge0$.

\begin{lemma}\label{Lem1.dten} Let $(u,p)$ be a sufficiently regular solution of problem \eqref{bf1}. Then, the following estimate holds:
%$$
\begin{equation}\label{1.dten}
\|v(t)\|^2_{L^2}+\int_t^{t+1}\|v(s)\|^2_{H^1}\,ds\le Q(\|u(0)\|_{H^2})e^{K t}+Q(\|g\|^2_{L^2})
\end{equation}
%$$
for some positive constant $K$ and monotone function $Q$.
\end{lemma}
\begin{proof} Multiplying equation \eqref{1.dteq} by $v(t)$, integrating over $\Omega$ and using that
$(f'(u)v)\cdot v\ge-K|v|^2, \forall u,v \in \R^3$ (see the condition
\eqref{0.fcond}), we arrive at
%$$
\begin{equation}\label{1.dtgr}
\Dt \|v(t)\|^2_{L^2}+\|v(t)\|_{H^1}^2\le 2K\|v(t)\|^2_{L^2}.
\end{equation}
Applying the Gronwall inequality to this estimate, we arrive at \eqref{1.dten} and finish the proof of the lemma.
\end{proof}

\begin{corollary} Let $(u,p)$ be a sufficiently smooth solution of the problem \eqref{bf1}. Then, the following estimate
holds:
%$$
\begin{equation}\label{1.h2div}
\|u(t)\|_{H^2}+\|\Nx p(t)\|_{L^2}\le Q(\|u(0)\|_{H^2})e^{Kt}+Q(\|g\|_{L^2})
\end{equation}
%$$
for some positive constant $K$ and monotone function $Q$ independent of $t$ and $u_0$.
\end{corollary}
Indeed, due to the control \eqref{1.dten}, we may rewrite equation \eqref{bf1} as an elliptic boundary value problem
%$$
\begin{equation}\label{1.ell}
\Dx w(t)-f(w(t))+\Nx p(t)=g_u(t):=-g+\Dt u(t)
\end{equation}
%$$
and apply the maximal regularity result of Theorem \ref{mr} (see Appendix) to that equation. Together with \eqref{1.dten} this
gives indeed estimate \eqref{1.h2div} and proves the corollary.
\par
We, however, note that the proved estimate \eqref{1.h2div} is {\it
divergent} as $t\to\infty$ and, by that reason, is not sufficient to
verify the {\it dissipativity} of the problem \eqref{bf1} in $H^2$.
In order to overcome this drawback, we need the $L^2\to H^2$
smoothing property for the solutions of \eqref{bf1}. This result
will be obtained exploiting the parabolic smoothing for equation
\eqref{1.dteq} together with the already established control
\eqref{1.dt} for $v(t)=\Dt u(t)$.

\begin{lemma}\label{Lem1.dtsm} Let $(u,p)$ be a sufficiently regular solution of the problem \eqref{bf1}. Then, the following
estimate holds:
%$$
\begin{equation}\label{1.dtsm}
\|\Dt u(t)\|_{L^2}\le \frac{1+t^3}{t^3}\(Q(\|u(0)\|_{L^2})e^{-\alpha t}+Q(\|g\|_{L^2})\),\ \ t>0,
\end{equation}
%$$
where the positive constant $\alpha$ and the monotone function $Q$ are independent of $t$ and $u$.
\end{lemma}
\begin{proof} We first note that, due to the energy estimate \eqref{1.en}, it is sufficient to verify \eqref{1.dtsm}
for $t\in(0,1]$ only.  To this end, we multiply \eqref{1.dtgr} by
$t^N$ (where the exponent $N$ will be specified later) and integrate
with respect to $t$. Then, we have
%$$
\begin{equation}\label{1.dtsm1}
\sup_{s\in[0,t]}\left\{s^N\|v(s)\|^2_{L^2}\right\}+\int_0^t s^N\|v(s)\|^2_{H^1}\,ds\le C\int_0^t s^{N-1}\|v(s)\|^2_{L^2}:=I(t),
 \end{equation}
%$$
where $C=C(N,K)$ is independent of $t$ and $u$.
\par
We estimate $I(t)$ using \eqref{1.dt} and the interpolation
inequality $\|v\|_{L^2}\le C\|v\|_{H^{-2}}^{1/3}\|v\|_{H^1}^{2/3}$:
%$$
\begin{multline}\label{1.est}
I(t)\le C\sup_{s\in[0,t]}\left\{s^{N/2}\|v(s)\|_{L^2}\right\}\int_0^ts^{N/2-1}\|v(s)\|_{L^2}\,ds\le
 \\\le \sup_{s\in[0,t]}\left\{s^{N/2}\|v(s)\|_{L^2}\right\}\int_0^t(s^{N/2}
 \|v(s)\|_{H^1})^{2/3}(s^{N/2-3}\|v(s)\|_{H^{-2}})^{1/3}\,ds
\le\\\le 1/2\sup_{s\in[0,t]}\left\{s^{N}\|v(s)\|_{L^2}^2\right\}+
1/2\int_0^t s^N\|v(s)\|^2_{H^1}\,ds+\\+C'\(\int_0^t
s^{N/2-3}\|v(s)\|_{H^{-2}}\,ds\)^2.
\end{multline}
%$$
Fixing now $N=6$, using the control \eqref{1.dt} in order to estimate the right-hand side of \eqref{1.est} and inserting
it into the right-hand side of \eqref{1.dtsm1}, we see that
%$$
\begin{equation}\label{1.est2}
\sup_{s\in[0,t]}\left\{s^{6}\|v(s)\|_{L^2}^2\right\}\le
1/2\sup_{s\in[0,t]}
\left\{s^{6}\|v(s)\|_{L^2}^2\right\}+Q(\|u(0)\|_{L^2})+Q(\|g\|_{L^2})
\end{equation}
(recall, we have assumed that $t\le1$).
It only remains to note that \eqref{1.est2} immediately gives \eqref{1.dtsm} for
$t\le1$. Lemma \ref{Lem1.dtsm} is proved.
\end{proof}
We summarize the obtained estimates in the following theorem.
\begin{theorem}\label{Th1.h2dis} Let $(u,p)$ be a sufficiently regular solution of the problem \eqref{bf1}.
Then, the following
estimate holds:
%$$
\begin{equation}\label{1.h2en}
\|u(t)\|_{H^2}+\|\Nx p(t)\|_{H^1}\le Q(\|u(0)\|_{H^2})e^{-\alpha t}+Q(\|g\|_{L^2}),
\end{equation}
where the positive constant $\alpha$ and a monotone function $Q$ are
independent of $t$ and $u$. Moreover, the following smoothing
property is valid:
%$$
\begin{equation}\label{1.h2sm}
\|u(t)\|_{H^2}+\|\Nx p(t)\|_{L^2}\le Q\(\frac{1+t^3}{t^3}\|u(0)\|_{L^2}\)e^{-\alpha t}+Q(\|g\|_{L^2}),\ \ t>0.
\end{equation}
%$$
\end{theorem}
Indeed, the estimate \eqref{1.h2sm} is an immediate corollary of
\eqref{1.dtsm} and the maximal elliptic regularity of Theorem
\ref{mr} applied to the elliptic equation \eqref{1.ell}. In order to
verify \eqref{1.h2en}, it is sufficient to use the {\it divergent}
in time estimate \eqref{1.h2div} for $t\le1$ and estimate
\eqref{1.h2sm} for $t\ge1$.

\section{Well-posedness and attractors}\label{s2}
The estimates obtained in the previous section, allow us to prove
the existence and uniqueness of a solution of the problem
\eqref{bf1} as well as to establish existence of the global
attractor for the associated semigroup.
 We start with the definition of a weak solution of that equation excluding the pressure in a standard way.

\begin{definition}\label{Def2.weak} A function
%$$
\begin{equation}\label{2.regweak}
u\in C([0,\infty),H)\cap L^2_{loc}([0,\infty),H^1)\cap L^{r+1}_{loc}([0,\infty),L^{r+1}(\Omega))
\end{equation}
%$$
is called a weak solution of \eqref{bf1} if it satisfies
\eqref{1.nopres} in the sense of distributions, i.e.,
$$
-\int_\R(u(t),\Dt\varphi(t))\,dt=-\int_\R(\Nx u(t),\Nx\varphi(t))-(f(u(t)),\varphi(t))+(g,\varphi(t))\,dt
$$
for all $\varphi\in C_0^\infty(\R_+\times\Omega)$ such that $\divv\varphi(t)\equiv0$.
\end{definition}

The next lemma establishes the uniqueness of a weak solution.

\begin{lemma}\label{Lem2.uni} Let the nonlinearity $f$ satisfy assumptions \eqref{0.fcond}. Then, the weak solution of problem
 \eqref{bf1} is unique.
 Moreover, for any two solutions $u_1(t)$ and $u_2(t)$ (with different initial data) of the  equation \eqref{bf1}, the following
estimate holds:
%$$
\begin{equation}\label{2.lip}
\|u_1(t)-u_2(t)\|_{L^2}\le e^{(K-\lambda_1)t}\|u_1(0)-u_2(0)\|_{L^2},
\end{equation}
%$$
where $K$ is the same as in \eqref{0.fcond} and $\lambda_1>0$ is the first eigenvalue of the  operator $A$.
\end{lemma}
\begin{proof} Let $u_1(t)$ and $u_2(t)$ be two different energy solutions of \eqref{bf1} and let $v(t):=u_1(t)-u_2(t)$. Then,
this function solves:
%$$
\begin{equation}\label{2.dif}
\Dt v= Av -\Pi(f(u_1)-f(u_2)),\ \ v(0)=u_1(0)-u_2(0).
\end{equation}
%$$
Note that, due to the regularity \eqref{2.regweak} of a weak solution and the growth restrictions on $f$, all terms in equation \eqref{2.dif} belong
to the space
$$
L^2([0,T],H^{-1})+L^{1+1/r}([0,T],L^{1+1/r}(\Omega))=[L^2([0,T],H^1)\cap L^{r+1}([0,T],L^{1+r}(\Omega))]^*.
$$
In particular, the function $t\to \|u(t)\|^2_{H}$ is absolutely continuous and
$$
\frac d{dt}\|u(t)\|_{L^2}^2=2(\Dt u(t),u(t)).
$$
Multiplying now equation \eqref{2.dif} by $v(t)$, integrating over
$\Omega$ and using the inequality
$$
(f(u_1)-f(u_2(t)).(u_1-u_2)\ge -K|u_1-u_2|^2,\ \ \forall u_1,u_2\in\R^3
$$
(due to the first assumption of \eqref{0.fcond}), we arrive at
%$$
\begin{equation}\label{2.di}
1/2\frac d{dt}\|v(t)\|_{L^2}^2\le K\|v(t)\|^2_{L^2}-(Av(t),v(t))\le (K-\lambda_1)\|v(t)\|^2_{L^2}
\end{equation}
and the Gronwall inequality now gives the uniqueness and estimate \eqref{2.lip}. Lemma \ref{Lem2.uni} is proved.
\end{proof}
We are now able to state our main result on the well-posedness and regularity of solutions of problem \eqref{bf1}.

\begin{theorem}\label{Th.reg} Let the nonlinearity $f$ satisfy assumptions \eqref{0.fcond} and let $g\in L^2(\Omega)$.
Then, for every $u_0\in H$, problem \eqref{bf1} possesses a unique weak solution $u$
 (in the sense of Definition \eqref{Def2.weak}). Moreover, $u(t)\in H^2$ for all $t>0$ and the estimate \eqref{1.h2sm} holds.
  In addition, if $u_0\in H^2$, the estimate \eqref{1.h2en} also holds.
\end{theorem}
\begin{proof} Indeed, the existence of a weak solution can be obtained in a standard way using, say, the Galerkin
 approximation method. The uniqueness is proved in Lemma \ref{Lem2.uni}. Thus, we only need to justify the estimates
 \eqref{1.h2sm} and \eqref{1.h2en}. To this end, we note that the estimates \eqref{1.dten} and \eqref{1.dtsm} for
  the differentiated equation \eqref{1.dteq} can be also first obtained on the level of the Galerkin approximations
  and then justified by passing to the limit (remind that the uniqueness of a weak solution holds). Finally,
  rewriting the
  problem \eqref{bf1} in the form of elliptic problem \eqref{1.ell} and using the  Theorem \ref{mr}, we justify the desired
   estimates  \eqref{1.h2sm} and \eqref{1.h2en}. Thus, Theorem \ref{Th.reg} is proved.
\end{proof}
 Thus, under the assumptions of Theorem \ref{Th.reg}, the Brinkman-Forchheimer problem \eqref{bf1} generates
  a dissipative semigroup $S(t)$ in the phase space $H$:
  %$$
  \begin{equation}\label{2.sem}
  S(t): H\to H,\ S(t)u_0:=u(t),
  \end{equation}
  %$$
where $u(t)$ solves \eqref{bf1} with $u(0)=u_0$. Our next task is to
verify the existence of a global attractor for that semigroup. For
the convenience of the reader, we start with reminding the
definition of the attractor, see
\cite{BV},\cite{Hale},\cite{LA3},\cite{Tem}  for more details.

\begin{definition}\label{Def2.attr} A set $\mathcal A\subset H$ is a global attractor of a semigroup $S(t): H\to H$
 if the following properties are satisfied:
 \par
 1) $\mathcal A$ is a compact subset of $H$;
 \par
 2) $\mathcal A$ is strictly invariant: $S(t)\mathcal A=\mathcal A$ for all $t\ge0$;
 \par
 3) It attracts the images of all bounded sets as time goes to infinity, i.e., for every bounded subset $B\subset H$ and every
 neighborhood $\mathcal O(\mathcal A)$ of $\mathcal A$, there exists $T=T(B,\mathcal O)$ such that
 $$
 S(t)B\subset \mathcal O(\mathcal A), \ \ \forall t\ge T.
$$
\end{definition}
The following theorem states the existence of the attractor for the
problem considered.
\begin{theorem}\label{Th2.attr} Let the assumptions of Theorem \ref{Th.reg} hold. Then the solution semigroup
\eqref{2.sem} associated with the Brinkman-Forchheimer equation \eqref{bf1} possesses a global attractor $\mathcal A$
 (in the sense of the above definition) which is bounded in $H^2$ and is generated by all complete bounded solutions of
  \eqref{bf1} defined for all $t\in\R$:
  %$$
  \begin{equation}\label{2.atrstr}
  \mathcal A=\mathcal K\big|_{t=0},
  \end{equation}
  %$$
  where $\mathcal K:=\{u\in C_b(\R,H^2), \ u $ solves \eqref{bf1}\}.
  \end{theorem}
  Indeed, according to the abstract attractor existence theorem (see e.g., \cite{BV},\cite{Tem}), we only need to check that
    the considered semigroup  is continuous with respect to the initial data (for every fixed $t$) and
    it possesses a compact absorbing set in $H$. But the first assertion is an immediate corollary of Lemma \ref{Lem2.uni} and
    the second one follows from the estimate \eqref{1.h2sm}. Moreover, this estimate gives the absorbing
     set bounded in $H^2$. Since the attractor is always contained in an absorbing set, we have verified the
      existence of a global attractor $\mathcal A$ which is bounded in $H^2$. Finally, the representation \eqref{2.atrstr}
       of the attractor in terms of completer bounded trajectories is also a standard corollary of the attractor
        existence theorem mentioned above.

  \begin{remark}\label{Rem2.reg} Although, we have stated only the $H^2$-regularity of the attractor $\mathcal A$,
   it can be further improved (if $f$, $\Omega$ and $g$ are smooth enough) using the maximal regularity for the {\it linear}
    Stokes equation and bootstrapping. In particular, if $f$, $\Omega$ and $g$ are  $C^\infty$ smooth, the attractor
    will be also $C^\infty$-smooth.
    \par
    Another standard corollary of the general theory is the fact that the obtained attractor has
    a finite Hausdorff and fractal dimension in $H$. The proof of this fact is a straightforward implementation of the
     volume contraction technique to our equation (see e.g., \cite{BV,Tem}). Indeed, due to the embedding $H^2\subset C$,
     the nonlinearity $f$ is subordinated to the linear part of the equation (no matter how large
     is the growth exponent $r$) and one even is able to reduce formally the problem considered to the case of abstract semilinear
      parabolic equations.
  \end{remark}

  To conclude this section, we discuss the particular case of \eqref{bf1} where
  %$$
  \begin{equation}\label{2.grad}
  f(u)=-\nabla_u F(u),\ \
  \end{equation}
  %$$
  for some scalar function $F\in C^2(\R^3)$. Note that this condition is satisfied for the "most natural"
  nonlinearities
  $$
  f(u)=a u|u|^{r-1}-b u.
  $$
   In that case, multiplying the equation by $\Dt u$
  and integrating over $\Omega$, we get
  $$
  \frac d{dt}\mathcal L(u(t))=-\|\Dt u(t)\|^2_{L^2}\le0,
  $$
  where
  $$
  \mathcal L(u):=\frac12(\Nx u,\Nx u)+(F(u),1).
  $$
  Thus, the solution semigroup $S(t)$ possesses the global Lyapunov functional $\mathcal L(u)$ and applying the standard
   arguments (see \cite{Hale},\cite{LA3}) to our problem, we obtain the following result.

  \begin{corollary} Let the assumptions of Theorem \ref{Th2.attr} and the condition \eqref{2.grad} be satisfied. Then,
  every trajectory  $u(t)$ stabilizes as $t\to\infty$ to the set of equilibria
  %$$
  \begin{equation}\label{2.eq}
  \mathcal R:=\{u_0\in H^2, Au_0-\Pi f(u_0)=\Pi g\}.
  \end{equation}
  %$$
  Furthermore, if the set $\mathcal R$ is discrete, every trajectory $u(t)$ converges to a single
   equilibrium $u_0\in\mathcal R$ and the rate of convergence is exponential if that equilibrium is hyperbolic.
   \end{corollary}

  \begin{remark}\label{Rem2.gen} Note that, for generic $g\in L^2$, the set $\mathcal R$ will
   contain only {\it hyperbolic} equilibria (see \cite{BV}). In that case, as it is not difficult to prove (again verifying
    the conditions of the abstract theorem on regular attractors stated in \cite{BV}), the attractor $\mathcal A$ can be presented as a finite union
    of finite-dimensional submanifolds of $H$ (the unstable manifolds of all equilibria) and that the rate of attraction of
     any bounded subset $B$ to the global attractor $\mathcal A$ is {\it exponential}.
  \end{remark}

\section{The convective Brinkman-Forchheimer equations}\label{s3}
In this section, we extend the results of the previous section to the case of the following Brinkman-Forchheimer
 equation with the Navier-Stokes type inertial term:
 %$$
 \begin{equation}\label{3.ns}
 \Dt u+(u,\Nx)u+\Nx p=\Dx u-f(u)+g,\ \ \divv u=0.
 \end{equation}
 %$$
Note that the case $f=0$ corresponds to the classical Navier-Stokes problem and the general case $f\ne0$ can
 be also considered as the so-called tamed Navier-Stokes equation, see \cite{RZh}.
\par
As before, the nonlinearity $f$ is assumed to satisfy conditions \eqref{0.fcond} but with the additional {\it lower}
bound $r>3$ which is necessary for the uniqueness. Note that no upper bounds for the growth exponent is posed.
\par
As before, we define a weak solution $u$ as a function of the class \eqref{2.regweak} satisfying \eqref{3.ns}
in the sense of distributions, see Definition \ref{Def2.weak}. In addition the assumption $r\ge3$ guarantees that
%$$
\begin{equation}\label{3.ine}
(u,\Nx)u\in L^{4/3}\subset L^q,\ \ q:=(r+1)^*\le 4/3
\end{equation}
%$$
and, therefore, in contrast to the case of the classical Navier-Stokes equations, the multiplication of \eqref{3.ns} by $u$
with integration over $\Omega$ is justified for {\it any} weak energy solution of that equation. Thus, we have
verified that any weak energy solution of \eqref{3.ns} satisfies the energy estimate \eqref{1.en}. The existence
of an energy solution can be then obtained in a standard way via the Galerkin approximation method.
\par
The next Lemma gives the uniqueness of the energy solution for the case $r>3$.

\begin{lemma}\label{Lemma3.uni} Let the nonlinearity $f$ satisfy \eqref{0.fcond} with $r>3$ and $g\in L^2$. Then,
for every $u_0\in H$, the problem \eqref{3.ns} possesses a unique
weak solution $u$ and this solution satisfies
 the energy estimate \eqref{1.en}.
\end{lemma}
\begin{proof}
Indeed, let $u_1$ and $u_2$ be two solutions and let $v=u_1-u_2$. Then, this function solves
%$$
\begin{equation}\label{3.dif}
\Dt v+(v,\Nx)u_1+(u_2,\Nx)v+\Nx q=\Dx v-[f(u_1)-f(u_2)],\ \ \divv v=0.
\end{equation}
%$$
Multiplying this equation by $v$, integrating by parts and using that $f$ satisfies \eqref{0.fcond}, we will have
$$
\frac d{dt}\|v\|^2_{L^2}+2\|\Nx v\|^2_{L^2}+\alpha(|u_1|^{r-1}+|u_2|^{r-1},|v|^2)\le C\|v\|^2_{L^2}+2|((v,\Nx) u_1,v)|
$$
for some positive $\alpha$ depending on $\kappa$ from
\eqref{0.fcond}. Here we have implicitly used  that the first
condition of \eqref{0.fcond} implies that
 $$
 (f(u_1)-f(u_2),u_1-u_2)\ge -C\|u_1-u_2\|^2_{L^2}+\alpha(|u_1|^{r-1}+|u_2|^{r-1},|u_1-u_2|^2),
 $$
 see \cite{Li} and \cite{GGZ} for the details.
 \par
The last term in the above differential inequality can be estimated integrating by parts once more and using that $r-1>2$:
%$$
\begin{multline}
2|((v,\Nx) u_1,v)|\le2(|u_1|\cdot|v|,|\Nx v|)\le \|\Nx v\|^2_{L^2}+C(|u_1|^2,|v|^2)\le\\\le\|\Nx v\|^2_{L^2}+
\alpha(|u_1|^{r-1}+|u_2|^{r-1},|v|^2)+C\|v\|^2_{L^2}.
\end{multline}
%$$
Thus, we have
%$$
\begin{equation}\label{3.uni}
\frac d{dt}\|v\|^2_{L^2}+\|\Nx v\|^2_{L^2}\le C\|v\|^2_{L^2}
\end{equation}
%$$
and the uniqueness is proved.
\end{proof}
\begin{remark}\label{Rem3.3} As we see from the proof, the uniqueness holds for the case $r=3$ if the
 coefficient $\kappa$ in \eqref{0.fcond} is {\it large} enough. However, we do not know whether or not
 the uniqueness holds for any cubic nonlinearity (without this assumption).
 \end{remark}
The next theorem is analogous to Theorem \ref{Th.reg} and gives the regularity of  solutions for problem \eqref{3.ns}.

\begin{theorem}\label{Th3.reg} Let the function $f$ satisfy \eqref{0.fcond} with $r>3$ and let $g\in L^2$. Then, for any $u_0\in H$,
the associated solution $u(t)$ of \eqref{3.ns} is more regular for $t>0$ ($u(t)\in H^2$) and estimate \eqref{1.h2sm} holds.
 In addition, if $u_0\in H^2$ then estimate \eqref{1.h2en} also holds.
\end{theorem}
\begin{proof} The proof of this theorem is also analogous to the proof of Theorem \ref{Th.reg}. Indeed, differentiating
equation \eqref{3.ns} with respect to $t$ and arguing as in the
proof of the previous lemma, we conclude that the function $v=\Dt u$
satisfies the differential inequality \eqref{3.uni}. On the other
hand, using \eqref{3.ine} for the control
 of the inertial term and arguing as in Corollary \ref{Cor1.dt}, we derive estimate \eqref{1.dt} and based on that
  estimate and inequality \eqref{3.uni} for $v=\Dt u$, one derives the controls \eqref{1.dtsm} and \eqref{1.dten} for
   the time derivative $v=\Dt u$ (all these estimates can be justified via the Galerkin approximations).
   \par
   Finally, having the control of the $L^2$-norm of $\Dt u$, one can treat problem \eqref{3.ns} as an elliptic
    boundary value problem of the form \eqref{nst1} and apply Corollary \ref{CorA.reg} which gives the desired
     estimate for the $H^2$-norm and finishes the proof of the theorem.
     \end{proof}
\begin{remark}\label{Rem3.rsmall} Note that the {\it nonlinear} localization technique used
 in the proof of Corollary \ref{CorA.reg} is not necessary if $r\le5$ where we may use the standard maximal regularity
  for the linear Stokes equation or in the case of periodic boundary condition. However, we do not know how to avoid
   these technicalities in a general case.
\end{remark}
Finally, let us note that the analogue of Theorem \ref{Th2.attr}
holds for the Navier-Stokes case as well.

\begin{theorem}\label{Th3.attr} Let the assumptions of Theorem \ref{Th3.reg} hold. Then the solution
 semigroup $S(t): H\to H$ possesses a global attractor $\mathcal A$ which is a bounded subset of $H^2$ and possesses the
 standard description \eqref{2.atrstr}.
\end{theorem}
The proof of this theorem repeats word by word the proof of Theorem \ref{Th2.attr} and so is omitted.
\begin{remark} \label{Rem.last} To conclude, we note that all assertions formulated in Remark \ref{Rem2.reg}
 remain true for the convective case as well.
\end{remark}

\section{Appendix: Maximal regularity for semi-linear Stokes problem}\label{s4}
The appendix is devoted to  the stationary problem associated with the problem
\eqref{bf1}, that is the following semi-linear Stokes problem:

\begin{equation}\label{st1}
\begin{cases}
-\Dx w+f(w)+\Nx p=g, \  \ \divv w=0, \  x \in \Omega, \\
w=0 , \ \ x \in \partial\Omega,\ \ \int_\Omega p(x)\,dx=0.
\end{cases}\end{equation}
Here $w=(w_1,w_2,w_3)$, $g\in L^2(\Omega)$ is a given function and
the nonlinearity $f$ satisfies assumptions \eqref{0.fcond} with
arbitrary $r>1$ and with $K=0$. Thus, we have assumed that the
nonlinearity $f$ is monotone $f'(u)\ge0$ and, therefore, the energy
solution of \eqref{st1} is unique.
\par
Our aim here is to prove the  $L^2$-maximal regularity
estimate for  problem \eqref{st1} (which is the non-linear version of the classical $L^2$-regularity theorem
for the Stokes operator). Before stating the main result, we first remind the straightforward $L^q$-regularity result where $q=(r+1)^*=1+\frac1r$.
\begin{lemma}\label{Lem-simple} Let the above assumptions on $f$ hold and let $g\in L^q(\Omega)$.
 Then, problem \eqref{st1} has a unique solution $(w,p)\in\mathcal F_q$ where
$$
\mathcal F_q:=\{(w,p),\, w\in W^{2,q}(\Omega)\cap L^{r+1}(\Omega),\ \ p\in W^{1,q}(\Omega)\}
$$
and the following estimate holds:
%$$
\begin{equation}\label{q-reg}
\|w\|_{W^{2,q}}+\|w\|_{L^{r+1}}^r+\|p\|_{L^{3/2+\eb}}+\|w\|_{H^1}^{2r/(r+1)}\le C(1+\|g\|_{L^q}),
\end{equation}
%$$
for some positive $C$ independent of $g$ and sufficiently small $\eb=\eb(q)>0$.
\end{lemma}
\begin{proof} We give below only the derivation of the estimate in the space $\mathcal F_q$
(the existence and uniqueness of the solution can be obtained in a standard way, e.g., using the
Galerkin approximation method). Indeed, multiplying equation \eqref{st1} by $w$, integrating by parts and using \eqref{0.fcond}, we arrive at
%$$
\begin{equation}\label{st3}
\|w\|_{H^1}^2+\|w\|_{L^{r+1}}^{r+1}\le C(1+\|g\|^q_{L^q}), \ q:=(r+1)^*=1+\frac1r,
\end{equation}
where $C$ is independent of $g$ and $w$. Together with conditions on $f$, this gives, in particular, that
%$$
\begin{equation}\label{st4}
\|f(w)\|_{L^q}\le C(1+\|g\|_{L^q}),
\end{equation}
%$$
for some (new) constant $C$.
\par
Rewriting now the problem \eqref{st1} as a linear Stokes problem
%$$
\begin{equation}\label{a.1}
-\Dx w+\Nx p=h_w:=g-f(w)
\end{equation}
%$$
and applying  the maximal $L^q$-regularity estimate for this linear
Stokes problem, we have
%$$
\begin{equation}\label{st5}
\|w\|_{W^{2,q}}+\|p\|_{W^{1,q}}\le C(1+\|g\|_{L^q}).
\end{equation}
%$$
In particular, due to Sobolev embedding theorem $W^{1,q}(\Omega)\subset L^s(\Omega)$
with $s:=\frac{3q}{3-q}>3/2$ (since $r>1$ and $q<1$)
%$$
\begin{equation}\label{st6}
\|p\|_{L^{3/2+\eb}}\le C(1+\|g\|_{L^q}),
\end{equation}
%$$
where $\eb=\eb(r)>0$ depends only on the exponent $r$.
\end{proof}
We are now ready to state the main result of this section.
\begin{theorem}\label{mr} Let $w$ be an energy solution of problem \eqref{st1}, $g\in L^2(\Omega)$ and the
 assumptions \eqref{0.fcond} on $f$  hold. Then, $w\in H^2(\Omega)$ and
the following estimate is valid:
\begin{equation}\label{st34}
\|w\|_{H^2(\Omega)}+\|p\|_{H^1(\Omega)}\le Q(\|(w,p)\|_{\mathcal F_q})(1+\|g\|_{L^2}^{2-\kappa})
\end{equation}
for some monotone function $Q$ and positive $\kappa=\kappa(r)$.
\end{theorem}

%By using variational method we can prove the following theorem (see,
%e.g.. \cite{KP})
%\begin{theorem}\label{kp} If  $g\in L^{(r+1)/r}(\Omega)$ then the problem
%\eqref{st1} has a unique energy solution
%$$
%w \in V\cap (L^{r+1}(\Omega))^3.
%$$
%\end{theorem}
\begin{proof} As before, we restrict ourselves to the formal derivation of the regularity estimate \eqref{st34}. The existence
of a solution can be verified in a standard way using, e.g., the Leray-Schauder fixed point theorem
\par
We will use the so-called nonlinear localization method and split the derivation of that estimate in several steps.

\par
\noindent {\bf  Step 1: Interior regularity.}

At this stage, we
obtain the {\it interior} $H^2$-regularity estimate for the solution $w$.
To this end,
 we multiply
equation \eqref{st1}  by $\sum_i\partial_{x_i}(\phi\partial_{x_i}w)$ where
$\phi$ is a proper nonnegative cut-off function which vanishes near
the boundary and equals one identically inside of the domain. To be more precise, we assume that $\varphi\in C_0^\infty(\R^3)$
is such that $0\le\varphi(x)\le1$, $\varphi(x)\equiv0$ if $x\in\R^3\backslash\Omega$ and $\varphi(x)\equiv1$ if $x\in\Omega_\nu$
where
$$
\Omega_\nu:=\{x\in\Omega, \ \dist(x,\partial\Omega)>\nu\}
$$
and $\nu>0$ is a sufficiently small number.
Moreover, without loss of generality, we may assume that
%$$
\begin{equation}\label{st8}
|\Nx\phi(x)|\le C_{\nu,\delta}\phi(x)^{1-\delta}
\end{equation}
%$$
for some $\delta>0$ which can be chosen arbitrarily small and the constant $C_{\nu,\delta}$ depending only on $\delta$, $\nu$
and the shape of $\Omega$.
 \par
 Thus, we may estimate the term with the Laplacian as follows:
 %$$
 \begin{multline}\label{a2}
 (\Dx
 w,\sum_i\partial_{x_i}(\varphi\partial_{x_i}w))=\sum_i(\Nx\partial_{x_i}w,\Nx(\varphi\partial_{x_i}w))\geq \kappa
 \sum_i(\varphi\Nx\partial_{x_i}w,\Nx\partial_{x_i}w)-\\-(\Nx\partial_{x_i}w,\Nx\varphi\cdot\partial_{x_i}w)\ge
 \frac12\kappa(\varphi,|D^2_x w|^2)-(|\Nx\varphi|^2\varphi^{-1},|\Nx w|^2)\ge\\
 \ge\frac14\kappa\|\varphi^{1/2}w\|^2_{H^2}-C\|w\|^2_{H^1}.
 \end{multline}
 %$$
This, together with the energy estimate \eqref{st3} for the
subordinated terms, gives the following estimate:
%$$
\begin{equation}\label{st7}
\|\phi^{1/2} w\|_{H^2}^2+(\phi f'(w)\Nx w,\Nx w)\le
C(1+\|g\|^2_{L^2})+|(p,\partial_{x_i}(\Nx \phi\cdot \partial_{x_i} w)|.
\end{equation}
%$$
Using again the energy estimate to control the subordinated terms and \eqref{st8} to control $|\Nx\varphi|$,
 the last term can be estimated as follows:
%$$
\begin{equation}\label{st9}
|(p,\partial_{x_i}(\Nx\phi\cdot\partial_{x_i} w)|\le
C(1+\|g\|_{L^2}^2)+\eb\|\phi^{1/2}
w\|_{H^2}^2+C_\eb\|\phi^{1/2-\delta}p\|_{L^2}^2,
 \end{equation}
%$$
where $\eb>0$ is arbitrary. In order to estimate the last term in
the right-hand side of this estimate, we use the H\"older inequality
with exponents $\frac3{2\alpha}$ and $\frac3{3-2\alpha}$
($\alpha=1/2-\delta$) in the following way:
%$$
\begin{equation}\label{a.4}
\int_\Omega\varphi^{2\alpha}|u|^2\,dx=\int_\Omega(\varphi|u|)^{2\alpha}|u|^{2(1-\alpha)}\,dx\le C\|\varphi u\|^{2\alpha}_{L^3}
\|u\|^{2(1-\alpha)}_{L^{\frac{6(1-\alpha)}{3-2\alpha}}}.
\end{equation}
%$$
Since
$\frac{6(1-\alpha)}{3-2\alpha}=\frac32(1+\frac{\delta}{1+\delta})$
and $2\alpha=1-2\delta<1$, fixing $\delta>0$ small enough that
$\frac{\delta}{1+\delta}\le \frac23\eb$ and using \eqref{st6} for
estimating the $L^{3/2+\eb}$-norm of $p$, we arrive at
\begin{equation}\label{st10}
\|\phi^{1/2-\delta}p\|_{L^2}^2\le C\|\phi p\|_{L^3}^{2\alpha}\|p\|^{2-2\alpha}_{L^{3/2+\eb}}\le Q(\|(w,p)\|_{\mathcal F_q})(1+\|\phi p\|_{L^3}).
\end{equation}
Thus, due to \eqref{st9},
%$$
\begin{equation}\label{st11}
\|\phi^{1/2} w\|^2_{H^2}+(\phi f'(w)\Nx w,\Nx w)\le C(1+\|g\|_{L^2}^2)+Q(\|(w,p)\|_{\mathcal F_q})(1+\|\phi p\|_{L^3}).
\end{equation}
%$$
 In order to estimate the term in the
right-hand side, we take the divergence from both sides of
\eqref{st1} and write out
\begin{equation}\label{st12}
\Dx p=-\divv f(w)+\divv g.
 \end{equation} Multiplying this equation
by $\phi$, we have
\begin{equation}\label{st13}
\Dx (\phi p)=2\Nx\phi\cdot\Nx p+\Dx \phi p-\phi \divv f(w)-\phi\divv
g:=h.
\end{equation}
Furthermore, due the growth assumptions \eqref{0.fcond} on $f$ and the energy estimate \eqref{st3},
\begin{multline}\label{st14}
\|\phi\divv f(w)\|_{L^{q}}^q\le \int_\Omega \phi^q|f'(w)\Nx w|^q\,dx\le
\int_\Omega |\phi f'(w)\Nx w\cdot\Nx w|^{q/2}\cdot|\phi f'(w)|^{q/2}\,dx\le\\
\le C(\phi f'(w)\Nx w,\Nx w)^{q/2}\|f'(w)\|_{L^{\frac q{2-q}}}^{q/2}= C(\phi f'(w)\Nx w,\Nx w)^{q/2}\|f'(w)\|^{q/2}_{L^{\frac{r+1}{r-1}}}\le\\
\le Q(\|(w,p)\|_{\mathcal F_q})(1+(\phi f'(w)\Nx w,\Nx w))^{q/2}.
\end{multline}
%$$
 Using this estimate together with the energy
estimate for the pressure, we arrive at
%$$
\begin{equation}\label{st15}
\|h\|_{L^{q}+H^{-1}}\le \beta (\phi f'(w)\Nx w,\Nx
w)+Q_\beta(\|(w,p)\|_{\mathcal F_q})(1+\|g\|_{L^2}),
\end{equation}
%$$
where the positive constant $\beta$ may be chosen arbitrarily small (the term $H^{-1}$
 appears due to the term $\varphi\divv g$, other terms belong to $L^q$).
 Finally, due to the maximal
regularity for the Laplacian together with the Sobolev embedding $W^{2,q}\subset L^s$ with $s=\frac{3q}{3-2q}>3$
 and $H^1\subset L^6$, we have
%$$
\begin{equation}\label{st16}
\|\phi p\|_{W^{2,q}+H^1}+\|\phi p\|_{L^{3+\eb}}\le \beta(\phi f'(w)\Nx
w,\Nx w)+Q_\beta(\|(w,p\|_{\mathcal F_q})(1+\|g\|_{L^2}),
\end{equation}
%$$
for some small positive $\eb=\eb(r)$ and $\beta$.
Inserting this estimate into
\eqref{st11} and fixing $\beta$ small enough, we conclude that
\begin{equation}\label{st17}
\|\phi^{1/2}w\|_{H^2}^2+(\phi f'(w)\Nx w,\Nx w)\le Q(\|(w,p)\|_{\mathcal F_q})(1+\|g\|_{L^2}).
\end{equation}
Using the embedding $H^2\subset C$ and the fact that $\phi$ equals
one identically inside of the domain, we deduce the desired interior
regularity
%$$
\begin{equation}\label{st18}
\|w\|_{H^2(\Omega_\nu)}+\|p\|_{H^1(\Omega_\nu)}\le
Q_\nu(\|(w,p)\|_{\mathcal F_q})(1+\|g\|_{L^2}),
\end{equation}
%$$
 where
$\nu>0$ is arbitrary. Thus, the interiror $H^2$-regularity estimate is proved.

\noindent {\bf Step 2:  Boundary regularity: tangent directions.}

We now  obtain the $H^2$-regularity in tangent directions near the
boundary. The standard approach here is to make the change of
variables which straighten the boundary locally in a small
neighborhood of a boundary point $x_0$ and then obtain the global
estimate using the proper partition of unity. However, in order to
avoid the complicated notations, we will use the alternative
equivalent approach working directly with the derivatives in
tangential directions. Namely, let $\tau_1=\tau_1(x)$ and
$\tau_2=\tau_2(x)$ be two smooth vector fields in $\bar\Omega$ which
are {\it linear independent} in a small neighborhood
$\bar\Omega\backslash\Omega_\nu$ and such that, for any
$x_0\in\partial\Omega_\nu$, $\nu\in[0,\nu_0]$ the vectors
$\tau_1(x_0)$ and $\tau_2(x_0)$ generate the tangent plane to
$\partial\Omega_\nu$.
 \par
 Being pedantic, such vector fields usually do not exist {\it globally}, but only locally (in a neighborhood of a
 fixed point $x_0\in\partial\Omega$). However, the plane-field spanned by the pair of vectors $(\tau_1,\tau_2)$ is well-defined globally
 and the {\it tangent gradient} $\nabla_{\tau}$ is also well-defined globally.
  Nevertheless, in slight abuse of rigoricity and in order to avoid the completely standard technicalities,
   we assume that the both vector fields $\tau_1$ and $\tau_2$  are globally defined in $\bar\Omega$.
\par
 Let $z(x):=\partial_\tau w(x):=\sum_{i=1}^3\tau^i(x)\partial_{x_i}w$ (where $\tau=\tau_1$ or $\tau=\tau_2$).
  Then, this function solves
\begin{equation}\label{st19}
\begin{cases}
-\Dx z+f'(w)z+\Nx (\partial_\tau p)=h_z,\\
h_z:=\partial_\tau g+
\partial_{x_i}(T_i(x)p+K_i(x)\Nx w)+L\Nx w+M(x)p,\\
\divv z=C(x)\Nx w,\ \ z\big|_{\partial\Omega}=0
\end{cases}
\end{equation}
for some smooth (matrix) functions $T,K,L,M,C$ independent of $p$
and $w$. Multiplying this equation by $z$ and using the energy
estimates \eqref{st5} and \eqref{st6} (in order to estimate the
subordinated terms) together with the facts that
$z\big|_{\partial\Omega}=0$ and $\divv z=C(x)\Nx w$, we deduce after
the simple estimates that
%$$
\begin{equation}\label{st20}
\|z\|^2_{H^1}+(f'(w)z,z)\le Q(\|(w,p)\|_{\mathcal F_q})(1+\|g\|^2_{L^2})+C\|\Nx
w\|_{L^3}\|\partial_\tau p\|_{L^{3/2}}+C\|p\|_{L^2}^2.
\end{equation}
%$$
Let us first estimate the $L^3$-norm of $\Nx w$. To this end, we need to use the interpolation in the spaces
with different regularity in tangent ($\tau_1$ and $\tau_2$) and normal ($n$) directions (here we need also the boundary $\partial\Omega$ to be
smooth enough). Indeed, since we control
the $W^{2,q}$-norm of $w$ (due to estimate \eqref{st5}), from the Sobolev's trace theorem, we have the control of the
traces of $\Nx w$ on the surfaces $\partial\Omega_\nu$ in the $W^{1/r,q}$-norm (remind that for small $\nu$, $\partial\Omega_\nu$
are uniformly smooth if the initial $\partial\Omega=\partial\Omega_0$ is smooth and the Sobolev trace theorem as well as the
interpolation theorems below work):
%$$
\begin{equation}\label{st21}
\|\Nx w\|_{C(\nu\in[0,\nu_0], W^{1/r,1+1/r}(\partial\Omega_\nu))}\le Q(\|(w,p)\|_{\mathcal F_q}).
\end{equation}
%$$
On the other hand, using the estimate \eqref{st20} for $\Nx
\partial_\tau w$, and the embedding $H^1\subset L^s$ for every $s<\infty$
for 2D domains $\Omega_\nu$, we conclude that
\begin{equation}\label{st22}
\|\Nx w\|_{L^2(\nu\in[0,\nu_0],L^s(\partial\Omega_\nu))}\le C_s\|z\|_{H^1},
\end{equation}
where the constant $C_s$ depends only on $s$.
\par
Using now the standard interpolation
$$
(L^{1+1/r}(\partial\Omega_\nu),L^s(\partial\Omega_\nu))_{2/3}=L^3(\partial\Omega_\nu)
$$
if $\frac13=\frac13\frac1{1+1/r}+\frac23\frac1s$, i.e., with
$s_r=2(1+r)$, we see that
%$$
\begin{multline}\label{a10}
\|\Nx w\|_{L^3(\Omega\backslash\Omega_{\nu_0})}=\|\Nx w\|_{L^3(\nu\in[0,\nu_0], L^3(\partial\Omega_\nu))}\le\\\le
C\|\Nx w\|_{L^\infty(\nu\in[0,\nu_0],L^{1+1/r}(\partial\Omega_\nu))}^{1/3}\|\Nx w\|^{2/3}_{L^2(\nu\in[0,\nu_0], L^{s_r}(\partial\Omega_\nu))}\le
Q(\|(w,p)\|_{\mathcal F_q})\|z\|_{H^1}^{2/3}.
\end{multline}
%$$
Using also the interior regularity \eqref{st18}, we conclude that
%$$
\begin{equation}\label{st23}
\|\Nx w\|_{L^{3}(\Omega)}\le
Q(\|(w,p)\|_{\mathcal F_q})(1+\|z\|^{2/3}_{H^1(\Omega)}).
\end{equation}
%$$
Let us now estimate the $L^{3/2}$-norm of $\partial_\tau p$. To this
end, we rewrite equation \eqref{st19} in the form
%$$
\begin{equation}\label{a11}
\Dx z-\Nx(\partial_\tau p)=h, \ \  z|_{\partial\Omega}=0,\ \divv z=C(x)\Nx w
\end{equation}
%$$
with
\begin{equation}\label{st25}
h=f'(w)z+\partial_\tau g+
\partial_{x_i}(T_i(x)p+K_i(x)\Nx w)+L(x)\Nx w+M(x)p
\end{equation}
and note that, due to the energy estimates \eqref{st5}, \eqref{st6}
and similar to \eqref{st14},
%$$
\begin{equation}\label{st26}
\|h\|_{H^{-1}+L^{q}}\le
Q(\|(w,p)\|_{\mathcal F_q})(1+(f'(w)z,z)^{1/2})+C\|g\|_{L^2}.
\end{equation}
%$$
Applying the  $H^{-1}\to H^1$ regularity and $L^q\to W^{2,q}$
regularity estimates for the linear non-homogeneous Stokes equation,
similar to \eqref{st16}, we conclude that
%$$
\begin{equation}\label{st27}
\|\partial_\tau p\|_{L^2+W^{1,q}}+\|\partial_\tau p\|_{L^{3/2+\eb}}\le
Q(\|(w,p)\|_{\mathcal F_q})(1+(f'(w)z,z)^{1/2})+C\|g\|_{L^2},
\end{equation}
%$$
where $\eb=\eb(r)>0$ is small enough.
\par
 Finally, we
need to estimate the $L^2$-norm of $p$. To this end, using the fact
that we have the control of the $W^{1,q}$-norm of the pressure $p$
and Sobolev trace theorems,  similar to \eqref{st21}, we have
%$$
\begin{equation}\label{st28}
\|p\|_{L^\infty(\nu\in[0,\nu_0], L^1(\partial\Omega_\nu))}\le
Q(\|(w,p)\|_{\mathcal F_q}).
 \end{equation}
%$$
 On the other hand, due to the embedding
$W^{1,3/2}\subset L^6$ for the 2D domains $\Omega_\nu$,
%$$
\begin{equation}\label{st29}
\|p\|_{L^{3/2}(\nu\in[0,\nu_0],L^6(\partial\Omega_\nu))}\le
C\|\partial_\tau p\|_{L^{3/2}}
\end{equation}
%$$
and, consequently, due to interpolation with exponent
($L^{9/4}=(L^1,L^6)_{2/3}=(L^\infty,L^{3/2})_{2/3}$) together with
already proved interior regularity (and the fact that $9/4>2$), we
deduce the estimate:
%$$
\begin{multline}\label{st30}
\|p\|_{L^2(\Omega)}\le \|p\|_{L^2(\Omega\backslash\Omega_{\nu_0})}+\|p\|_{L^2(\Omega_{\nu_0})}\le\\\le
  Q(\|(w,p)\|_{\mathcal F_q})(1+\|g\|_{L^2}+\| p\|_{L^{9/4}(\nu\in[0,\nu_0],L^{9/4}(\partial\Omega_\nu))})
\le\\\le Q(\|(w,p)\|_{\mathcal F_q})(1+\|g\|_{L^2}+\|p\|_{L^\infty(\nu\in[0,\nu_0], L^1(\partial\Omega_\nu))}^{1/3}\|p\|^{2/3}_{L^{3/2}(\nu\in[0,\nu_0],L^6(\partial\Omega_\nu))})\le\\\le
Q(\|(w,p)\|_{\mathcal F_q})(1+\|g\|_{L^2}+\|\partial_\tau p\|_{L^{3/4}}^{2/3})\le\\\le
Q(\|(w,p)\|_{\mathcal F_q})(1+\|g\|_{L^2}+(f'(w)z,z)^{1/3}).
\end{multline}
%$$
Inserting now estimates
\eqref{st23},\eqref{st27} and \eqref{st30} into the right-hand side
of (\eqref{st20}, we finally arrive at
%$$
\begin{equation}\label{st31}
\|z\|^2_{H^1}+\|\Nx w\|_{L^3}^2+(f'(w)z,z)+\|p\|^2_{L^2}\le Q(\|(w,p)\|_{\mathcal F_q})(1+\|g\|_{L^2}^2)
\end{equation}
%$$
and the $H^2$-regularity of $w$ in tangent directions is verified. Now, we note that
$$
\|f'(w)z\|_{H^{-1}}\le C\|f'(w)z\|_{L^{6/5}}\le C(|f'(w)|z,z))^{1/2}\|f'(w)\|_{L^{5/4}}^{1/2}
$$
and, therefore, applying the $H^{-1}$-regularity theorem to the linear Stokes problem \eqref{a11} and using \eqref{st31}, we arrive at
%$$
\begin{equation}\label{est3}
\|\partial_\tau p\|_{L^2}\le Q(\|(w,p)\|_{\mathcal F_q})(1+\|g\|_{L^2})(1+\|f'(w)\|_{L^{5/4}}^{1/2}).
\end{equation}
%$$
\par
\noindent {\bf Step 3: Regularity in normal direction and the final estimate.}
Let us now multiply equation \eqref{st1} by $\Dx w$. Then, after integration by parts, we get
%$$
\begin{equation}\label{est1}
\|w\|^2_{H^2}+(|f'(w)|\Nx w,\Nx w)\le C(1+\|g\|^2)+C|(\Nx p,\Dx w)|.
\end{equation}
%$$
In addition, the second term in the right-hand side gives
%$$
\begin{multline*}
(f'(w)\Nx w,\Nx w)\ge \kappa(|w|^{r-1},|\Nx w|^2)\ge \kappa_1\|\Nx(|w|^{(r+1)/2})\|^2_{L^2}\ge\\
\ge
\kappa_3\|w\|^{r+1}_{L^{3(r+1)}}\ge\kappa_4\|f(w)\|_{L^3}^{r/(r+1)}-C
\end{multline*}
%$$
and, consequently, using the energy estimate and the interpolation
$$
\|f(w)\|_{L^2}\le C\|f(w)\|_{L^1}^{1/4}\|f(w)\|_{L^3}^{3/4},
$$
we will have
%$$
\begin{equation}\label{est2}
\|w\|^2_{H^2}+\|f(w)\|_{L^2}^{4(r+1)/(3r)}\le C(1+\|g\|^2)+C|(\Nx p,\Dx w)|.
\end{equation}
%$$
Thus, we only need to estimate the last term in the right-hand side of this inequality.
 To this end, we split the tangential and normal derivatives in that term and use \eqref{st31} and \eqref{est3} for estimating the
 tangential derivatives:
%$$
\begin{multline}\label{est4}
|(\Nx p,\Dx w)|\le |(\partial_\tau p,(\Dx w)_\tau)|+|(\partial_n p,(\Dx w)_n)|\le \frac14\|w\|^2_{H^2}+\\+C\|\partial_\tau p\|_{L^2}^2+
|(\partial_n p,\partial_n^2 w_n)|+C(|\Nx p|,|\Nx\partial_\tau w|+|\Nx w|)\le \frac14\|w\|^2_{H^2}+\\+
Q(\|(w,p)\|_{\mathcal F_q})[1+(1+\|g\|_{L^2}^2)\|f(w)\|_{L^{5/4}}+\|\Nx p\|_{L^2}(1+\|g\|_{L^2})]+\\+C\|\Nx p\|_{L^2}\|\partial_n^2w_n\|_{L^2}.
\end{multline}
%$$
To estimate the last term in the right-hand side of \eqref{est4}, we use that $w$ is divergent free and, therefore,
$$
\partial_n w_n+\partial_{\tau_1} w_{\tau_1}+\partial_{\tau_2}w_{\tau_2}=C(x)w.
$$
Differentiating that equation in the direction of the normal vector field, we obtain the estimate
$$
\|\partial_n^2w_n\|_{L^2}\le C\|\Nx\partial_\tau w\|_{L^2}+C\|\Nx w\|_{L^2}.
$$
Inserting that estimate to the right-hand side of \eqref{est4} and \eqref{est2} and using \eqref{st31} together with the obvious estimate
$$
\|\Nx p\|_{L^2}\le C\|f(w)\|_{L^2}+C\|g\|_{L^2}
$$
(which follows from the $L^2$-maximal regularity for the linear Stokes equation) and the interpolation
$$
\|f(w)\|_{L^{4/5}}\le C\|f(w)\|_{L^1}^{3/5}\|f(w)\|_{L^2}^{2/5},
$$
we arrive at
%$$
\begin{multline}\label{est5}
\|w\|^2_{H^2}+\|f(w)\|_{L^2}^{4(r+1)/(3r)}\le\\\le Q(\|(w,p)\|_{\mathcal F_q})
[1+(1+\|g\|_{L^2}^2)\|f(w)\|_{L^{2}}^{2/5}+\|f(w)\|_{L^2}(1+\|g\|_{L^2})].
\end{multline}
%$$
Finally, thanks to Young inequality, we derive from \eqref{est5} that
$$
\|w\|_{H^2}\le Q(\|(w,p)\|_{\mathcal F_q})(1+\|g\|_{L^2}^{2-\kappa}),
$$
where $\kappa=\kappa(r)>0$ is a positive number. Thus, the theorem is proved.
% Estimate \eqref{st31} allows us to finish the proof
% of the maximal regularity result {\it without} studying the
% regularity of the normal derivative. Indeed, \eqref{st31} together with the Sobolev embedding $W^{1,3}\subset
% L^{2r}$ give
%$$
%\begin{equation}\label{st32}
%\|w\|_{L^{2r}(\Omega)}\le Q(\|g\|_{L^2})
%\end{equation}
%$$
% and,
% consequently,
%$$
% \begin{equation}\label{st33}
% \|f(w)\|_{L^2(\Omega)}\le Q(\|g\|_{L^2}).
%\end{equation}
%$$
% Combining this estimate with the standard $L^2$-maximal
% regularity for the linear Stokes operator, we end up with estimate \eqref{st34} and finish the proof of the theorem.
\end{proof}
\begin{remark} Clearly,
when the nonlinear terms has the growth
$$
|f(u)|\leq C(1+|u|^3), \forall u \in \R^3,
$$
the maximal regularity estimate \eqref{st34} follows directly from
the energy estimate \eqref{st3} and the regularity estimate for the
linear Stokes problem. Note also that, in the case of periodic
boundary conditions the simple multiplication of the initial problem
by $\Dx w$ gives better estimate
$$
\|w\|_{H^2}\le C(1+\|g\|_{L^2}).
$$
However, for the case of Dirichlet boundary conditions, the additional (uncontrollable) boundary terms
 appear under the integration by parts, and we unable to obtain the $H^2$-regularity estimate which is linear with respect to
  the $L^2$-norm of $g$. However, as we will see below, the sub-quadratic growth rate of that estimate with respect to $g$ is
   enough to be able to apply it for the Navier-Stokes-type problem.
\end{remark}
To be more precise, we want to apply the above result to the following analogue of problem \eqref{st1} perturbed by the Navier-Stokes inertial term:
%$$
\begin{equation}\label{nst1}
\begin{cases}
-\Dx w+(w,\Nx)w+f(w)+\Nx p=g, \  \ \divv w=0, \  x \in \Omega, \\
w=0 , \ \ x \in \partial\Omega,\ \ \int_\Omega p(x)\,dx=0.
\end{cases}
\end{equation}
%$$
Indeed, since the inertial term vanishes after the multiplication the equation by $w$ and integrating over $x$, arguing as in Lemma \ref{Lem-simple},
we have:
$$
\|w\|_{H^1}^2+\|w\|^{r+1}_{L^{r+1}}\le C(1+\|g\|_{L^q}^q).
$$
In addition,
%$$
\begin{equation}\label{in32}
\|(w,\Nx)w\|_{L^{3/2}}\le C\|w\|_{L^6}\|\Nx w\|_{L^2}\le Q(\|g\|_{L^2}),
\end{equation}
%$$
for some monotone function $Q$. Thus, the $L^q$-norm of the inertial term is under the control if $q\le 3/2$ (= $r\ge2$) and,
applying the $L^q$-regularity estimate for the linear Stokes problem, we also have the control of the $W^{1,q}$-norm of $p$.
So, we have proved that
%$$
\begin{equation}\label{ns-simple}
\|(w,p)\|_{\mathcal F_q}\le Q(\|g\|_{L^q})
\end{equation}
%$$
if $r\ge2$. The next Corollary gives the $H^2$-regularity estimate
for the problem\eqref{nst1}.
\begin{corollary}\label{CorA.reg} Let $f$ satisfy \eqref{0.fcond} with $r\ge2$ and $g\in L^2$. Then, any
energy solution $(w,p)\in\mathcal F_q$ of problem \eqref{nst1} belongs to $H^2\times H^1$ and the following estimate holds:
%$$
\begin{equation}\label{ns-reg}
\|w\|_{H^2}+\|p\|_{H^1}\le Q(\|g\|_{L^2})
\end{equation}
%$$
for some monotone function $Q$.
\end{corollary}
\begin{proof} Let us first formally deduce {\it a priori} estimate \eqref{ns-reg}. To this end, we interpret
 the inertial term as an external force and apply estimate \eqref{st34}. Then, using \eqref{ns-simple}, we have
 %$$
\begin{equation}\label{ns-est1}
 \|w\|_{H^2}\le Q(\|g\|_{L^2})(1+\|(w,\Nx)w\|_{L^2}^{2-\kappa})
 \end{equation}
 %$$
 for some monotone $Q$ and some positive $\kappa=\kappa(r)$. Thus, we only need to estimate the $L^2$-norm
  of the inertial term. To this end, we use \eqref{in32} together with the interpolation inequalities and the fact that
  $H^2\subset W^{1,6}$:
 %$$
 \begin{multline}
 \|(w,\Nx)w\|_{L^2}\le \|(w,\Nx)w\|_{L^{3/2}}^{2/3}\|(w,\Nx)w\|_{L^6}^{1/3}\le\\\le
  Q(\|g\|_{L^2})\|w\|_{L^\infty}^{1/3}\|w\|_{H^2}^{1/3}\le
 Q_1(\|g\|_{L^2})\|w\|_{H^1}^{1/6}\|w\|_{H^2}^{1/2}\le Q_2(\|g\|_{L^2})\|w\|_{H^2}^{1/2}.
 \end{multline}
 %$$
 Inserting this estimate in the right-hand side of \eqref{ns-est1}, we deduce the desired a priori estimate \eqref{ns-reg}.
 \par
 The existence of a solution $(w,p)\in H^2\times H^1$ can be obtained in a standard way based on that estimate and
  approximating, for instance,  the growing non-linearity $f(w)$ by a sequence $f_n(w)$ of globally bounded ones (see, e.g.,
   \cite{KZ} for the details). However, since the solution of \eqref{nst1} may be not unique, we still need to verify that any
   energy solution of that equation satisfies estimate \eqref{ns-reg}.
   \par
   Indeed, let $w$ be an energy solution of \eqref{nst1}. Let us consider the following
   modified equation \eqref{nst1}:
   %$$
\begin{equation}\label{nst2}
\begin{cases}
-\Dx v+(v,\Nx)v+f(v)+Rv+\Nx p=g_w, \  \ \divv v=0, \  x \in \Omega, \\
v=0 , \ \ x \in \partial\Omega,\ \ \int_\Omega p(x)\,dx=0.
\end{cases}
\end{equation}
%$$
with $g_w:=g+Rw$. We claim that the solution $v=w$ of that equation is
 {\it unique} in the class of energy solutions if $R\gg1$ is large enough (this fact can be easily verified
 using the standard energy estimates). On the other hand, arguing as before, we can construct a regular solution
$(v,p)\in H^2\times H^1$ of that equation satisfying \eqref{ns-reg}. The uniqueness guarantees then that the initial solution
$u$ is also regular and satisfies this estimate. Thus, the corollary is proved.
 \end{proof}
\begin{remark} Of course, the condition $r\ge2$ is not necessary for the validity of Corollary \ref{CorA.reg}. Indeed, for
$r<2$, the $L^2$ (and even $L^3$) norm of nonlinearity $f(w)$ is under the control due to the initial energy estimate
 and the desired $H^2$-regularity can be obtained treating the nonlinearity $f$ as the external force (exactly as in the
  case of the classical stationary Navier-Stokes problem, see e.g. \cite{La}). However, we are mainly interested by
   the case of large $q$ where the above presented {\it nonlinear} localization technique becomes unavoidable.
 \end{remark}

\end{document}